\documentclass[12pt]{article}
\usepackage[a4paper, margin=2.3cm]{geometry}
\usepackage{amsmath,amssymb,amsthm}
\usepackage{color}
\usepackage{parskip}
\usepackage{setspace}

\usepackage{graphicx}
\usepackage[nice]{nicefrac}
\usepackage{amssymb}
\usepackage{multirow}
\usepackage{array}

\usepackage{amsmath,amsthm,amscd,amssymb, mathrsfs}

\usepackage{latexsym}

\numberwithin{equation}{section}

\theoremstyle{plain}
\newtheorem{theorem}{Theorem}[section]
\newtheorem{lemma}[theorem]{Lemma}

\newtheorem{proposition}[theorem]{Proposition}
\newtheorem{conjecture}[theorem]{Conjecture}

\theoremstyle{definition}

\theoremstyle{remark}

\newtheorem{case[theorem]}{Case}

\begin{document}
\title{On restriction estimates for the zero radius sphere \\over finite fields}
\author{Alex Iosevich \thanks{Department of Mathematics, University of Rochester New York. Email: iosevich@math.rochester.edu}\and  Doowon Koh \thanks{Department of Mathematics, Chungbuk National University. Email: koh131@chungbuk.ac.kr}\and Sujin Lee \thanks{Department of Mathematics, Chungbuk National University. Email: sujin4432@chungbuk.ac.kr} \and 
Thang Pham \thanks{Department of Mathematics, University of Rochester New York. Email: vpham2@math.rochester.edu}\and Chun-Yen Shen \thanks{Department of Mathematics, National Taiwan University. Email address: cyshen@math.ntu.edu.tw}}
\maketitle
\begin{abstract}  In this paper, we solve completely the $L^2\to L^r$ extension conjecture for the zero radius sphere over finite fields. 
We also obtain the sharp $L^p\to L^4$ extension estimate for non-zero radii spheres over finite fields, which improves significantly the previous result of the first and second listed authors. 
\end{abstract}

\section{Introduction} 
In the last few decades, much attention has been given to the Fourier restriction/extension problem for various surfaces, in part because it is closely related to 
questions about partial differential equations as well as problems in geometric measure theory such as the Kakeya conjecture. Given a surface $V$ in $\mathbb R^d$ endowed with surface measure $d\sigma$,
the extension problem for $V$ is to determine all pairs of exponents $(p,r)$ with $1\le p,r\le \infty$ such that the extension inequality
$$ \|(fd\sigma)^\vee\|_{L^r(\mathbb R^d)} \le C \|f\|_{L^p(V,\,d\sigma)}$$
holds for all functions $f \in L^p(V,\,d\sigma).$
By duality, the extension inequality above is equivalent to the following restriction inequality
$$ \|\widehat{g}\|_{L^{p^\prime}(V,d\sigma)} \le C \|g\|_{L^{r^\prime}(\mathbb R^d)},$$
where $p^\prime$ and $r^\prime$ denote usual H\"{o}lder's conjugates of $p$ and $r,$ respectively (i.e. $p^\prime=p/(p-1)$ and $r^\prime=r/(r-1)$).  Since this problem was initiated by Stein \cite{St78},  lots of deep results have been established but the problem remains open in higher dimensions. We refer the reader to \cite{Zy74, Ba85, Wo01, Ta04, Ta03, Gu15} for a more detailed description and recent developments on the Euclidean restriction conjecture.

In 2002, Mockenhaupt and Tao \cite{MT04} initially studied the finite field analogue of the Fourier restriction/extension problem for various algebraic varieties. In this introduction we review the definition and state our main results on this problem for spheres.

Let $\mathbb F_q^d$ be the $d$-dimensional vector space over a finite field $\mathbb F_q$ with $q$ elements. Throughout this paper we assume that $q$ is an odd prime power. We endow the space $\mathbb F_q^d$ with a counting measure $dm$. We write $\mathbb F_{q*}^d$ for the dual space of $\mathbb F_q^d$  which is endowed with the normalized counting measure $dx.$ Given an algebraic variety $V$ in $\mathbb F_{q*}^d$ we endow it with the normalized surface measure $d\sigma$ which assigns a mass of $|V|^{-1}$ to each point in $V.$ Notice that the surface measure $d\sigma$ on $V$ can be written by $\frac{q^d}{|V|} 1_{V} dx$, a measure on $\mathbb F_{q^*}^d,$ namely,
$$ d\sigma= \frac{q^d}{|V|} 1_{V} dx.$$ 
Hence, we can view the measure $d\sigma$ on $V$ as a function $\frac{q^d}{|V|} 1_{V}$ on $\mathbb F_{q^*}^d.$
Let $\chi$ denote the canonical additive character of $\mathbb F_q$. We  recall that the orthogonality relation of $\chi$ states that
 $$ \sum_{m\in \mathbb F_q^d} \chi(x\cdot m) 
 = \left\{  \begin{array}{ll} 0 \quad&\mbox{if}~~x\ne (0,\ldots,0) \\
q^d \quad&\mbox{if}~~ x=(0,\ldots,0).\end{array}\right.$$

 Given a complex-valued function $g: \mathbb F_q^d \to \mathbb C,$ the Fourier transform of $g$ is defined by
$$ \widehat{g}(x)=\sum_{m\in \mathbb F_q^d} g(m) \chi(-x\cdot m).$$ 
Let $f: \mathbb F_{q*}^d \to \mathbb C$ be a complex-valued function on the dual space $\mathbb F_{q*}^d$ of $\mathbb F_q^d.$ Then the inverse Fourier transform of $f$ is defined by
$$ f^{\vee} (m)= \frac{1}{q^d} \sum_{x\in \mathbb F_{q*}^d} f(x) \chi(m\cdot x).$$
Moreover, the inverse Fourier transform of the measure $f d\sigma$  is given by
$$ (fd\sigma)^{\vee} (m)= \frac{1}{|V|} \sum_{x\in V} f(x) \chi(m\cdot x)$$
where $d\sigma$ denotes the normalized surface measure on an algebraic variety $V\subset \mathbb F_{q*}^d.$ \\

Since the space $\mathbb F_q^d$ is isomorphic to its dual space $\mathbb F_{q*}^d$ as an abstract group, we  shall simply write $\mathbb F_q^d$ for $\mathbb F_{q*}^d$ in the case when the measure on $\mathbb F_{q*}^d$  does not play an important role.  For example, to denote a sum over $x \in \mathbb F_{q*}^d$,  we write $\sum_{x\in \mathbb F_q^d}$ for $\sum_{x\in \mathbb F_{q*}^d}.$ It is clear that 
 $$f(\alpha)=\widehat{(f^\vee)}(\alpha)=\sum_{\beta\in \mathbb F_q^d} f^\vee(\beta) \chi(-\alpha\cdot \beta),$$
 which is called the Fourier inversion formula. Expanding $|\widehat{g}(x)|^2$ by the definition of the Foureir transform, and using the orthogonality relation of $\chi,$ one can easily deduce  the following equation: 
$$ q^{-d}\sum_{x\in \mathbb F_q^d} |\widehat{g}(x)|^2 = \sum_{m\in \mathbb F_q^d} |g(m)|^2.$$
 This formula is refered to as the Plancherel theorem, and it can be rewritten as
$$ \|\widehat{g}\|_{L^2(\mathbb F_{q*}^d, dx)}=\|g\|_{L^2(\mathbb F_q^d, dm)}.$$
With notation above, the Fourier extension problem for $V$ is to determine  exponents $1\le p, r\le \infty$ such that the extension inequality
\begin{equation}\label{extensionDef} \| (fd\sigma)^{\vee}\|_{L^r(\mathbb F_q^d, dm)} \ll \|f\|_{L^p(V, d\sigma)}\end{equation}
holds for any functions $f$ on $V$ with the operator norm  independent of the size of the underlying finite field $\mathbb F_q.$ 

Here, and throughout the paper, we use $X \ll Y$ or $Y\gg X$ to denote that there exists a constant $C>0$ independent of $q$ such that $ X\le C Y.$ In addition, $X\sim Y$ means that $X\ll Y$ and $Y\ll X.$ 

We use the notation $R_V^*(p\to r)\ll 1$ to indicate that the above extension estimate \eqref{extensionDef} holds. By duality (see Theorem \ref{ThmDual} in Appendix), the extension estimate \eqref{extensionDef} is equivalent to the following Fourier restriction estimate:
$$ \| \widehat{g}\|_{L^{p'}(V, d\sigma)} \ll \|g\|_{L^{r'}(\mathbb F_q^d, dm)}.$$
Necessary conditions for $R_V^*(p\to r)\ll 1$ can be given in terms of $|V|$ and  the cardinality of an affine subspace $H$ lying on $V.$  Indeed, Mockenhaupt and Tao  \cite{MT04}  showed that if $V\subset \mathbb F_{q*}^d$ with $|V|\sim q^{d-1}$ and it  contains an affine subspace $H$ with $|H|=q^k,$ then the necessary conditions are given by 
\begin{equation}\label{Necessary2}
r\geq \frac{2d}{d-1}  \quad \mbox{and} \quad r\geq\frac{p(d-k)}{(p-1)(d-1-k)}.\end{equation}

When the variety $V\subset \mathbb F_{q*}^d$ is  either a paraboloid  or a sphere, the restriction/extension problem has received much attention. 
Here we recall that the paraboloid $P$ in $\mathbb F_{q*}^d$ is defined by
$$ P:=\{x\in \mathbb F_{q*}^d: x_1^2+\cdots+x_{d-1}^2=x_d\}$$
and the sphere $S_j, j\in \mathbb F_q,$ is defined by 
\begin{equation}\label{DefS} 
S_j:=\{x\in \mathbb F_{q*}^d: x_1^2 +\cdots+ x_d^2=j\}.
\end{equation}
The restriction/extension conjectures for curves on $\mathbb F_{q*}^2$ were completely solved. For example, Mockenhaupt and Tao \cite{MT04} established  $R^*_P(2\to 4)\ll 1$ for the parabola, which implies the conjecture, and the second and fifth listed authors \cite{KS12} showed that the result holds for any curve which does not contain a line. However, the finite field restriction problem for varieties in higher dimensions has not been solved.

In the finite field setting, it would be very interesting and hard to establish the sharp $L^2\to L^r$ or $L^p\to L^4$ estimate for the sphere or the paraboloid in even dimensions, because the conjectured results in those cases are much better than the Stein-Tomas result, the $L^2\to L^{(2d+2)/(d-1)}$ bound, which can be obtained by using the Stein-Tomas argument. It was shown by Mockenhaupt and Tao \cite{MT04} that  an $L^2\to L^r$ extension result for the paraboloid can be obtained from a consequence of an $L^p\to L^4$ extension estimate. 
Developing their method with the optimal $L^p\to L^4$ extension estimate, the authors of \cite{LL10} obtained the $L^2\to L^{2d^2/(d^2-2d+2) }$ extension estimate for the paraboloid in even dimensions. This result is much better than the Stein-Tomas result. Further improvement was made by Lewko \cite{Le13} who reduced the $L^2\to L^r$ extension problem to the estimation of the additive energy 
$$E(A):=|\{(x,y,z,w)\in A^4: x+y=z+w\}|$$ for $A\subset P.$ 

In recent papers \cite{IKL17, RS18, kov}, improving this additive energy is the crucial step to get new results on $L^2$-extension theorems for the paraboloid. For example, in a recent work \cite{RS18}, with a new additive energy bound, Rudnev and Shkredov followed the scheme in \cite{IKL17} to obtain the sharp $L^2\to L^{3}$ extension estimate for the paraboloid in the dimension $4$. For any $d\ge 8$ even, the sharp $L^2\to L^{(2d+4)/d}$ extension estimate for the paraboloid was obtained by the first, second listed authors and Lewko in an earlier work \cite{IKL17}.
\subsection*{The focus of this article}
In this paper we develop the restriction/extension theory for spheres $S_j$ in $\mathbb F_{q*}^d.$ 

Let $S_0\subset \mathbb F_{q^*}^d$ be the zero radius sphere.
As mentioned in \eqref{Necessary2},  necessary conditions for $R_{S_0}^*(p\to r)$ bound can be determined by the size of a subspace $H$ lying in $S_0.$ In particular, if $|H|=q^{\ell},$ then a necessary condition for $R^*_{S_0}(2\to r)$ bound is given by 
\begin{equation}\label{necea}  r\geq\frac{2(d-\ell)}{(d-1-\ell)}.\end{equation}
 The following lemma is well known.
 \begin{lemma} [\cite{Vi12}]\label{clem} Let $H$ be a maximal subspace lying on the zero radius sphere $S_0\subset \mathbb F_q^d.$  Then the following statements hold:
\begin{enumerate} 
\item If $d=4k-2, k\in \mathbb N,$ and $q\equiv 3 \mod{4},$ then    $|H|=q^{\frac{d-2}{2}}.$
\item If $d= 4k$ (or $d=4k-2$ with $q\equiv 1\mod{4}),$ then 
$|H|=q^{\frac{d}{2}}.$
\item If $d\ge 3$ is odd, then $|H|=q^{(d-1)/2}.$
\end{enumerate}
\end{lemma}

 The following is the conjecture of the sharp $L^2\to L^r$ extension estimate for the zero radius sphere, which can be easily made from (\ref{necea}) and Lemma \ref{clem}. 
 
\bigskip

\begin{conjecture}\label{conjzero}
Let $S_0$ be the zero radius sphere in $\mathbb F_q^d.$
Then the following statements hold:
\begin{enumerate} 
\item If $d=4k+2, k\in \mathbb N,$ and $q\equiv 3 \mod{4},$ then    $R^*_{S_0}\left(2\to \frac{2d+4}{d}\right)\ll 1.$
\item If $d= 4k$ (or $d=4k-2$ with $q\equiv 1\mod{4}),$ then 
$R^*_{S_0}\left(2\to \frac{2d}{d-2}\right)\ll 1.$
\item If $d\ge 3$ is odd, then $R^*_{S_0}\left(2\to \frac{2d+2}{d-1}\right)\ll 1.$
\end{enumerate}
\end{conjecture}

The statements $(2)$ and $(3)$ of the above conjecture were already established by using the finite field Stein-Tomas argument (see Theorem 2.1 in \cite{KohShen13}).
In fact, the following sharp Fourier decay estimates on the surface measure $d\sigma$ on $S_0$ were mainly used to prove them:
$$ \max_{m\in \mathbb F_q^{d}\setminus \{\vec{0}\}}|(d\sigma)^{\vee}(m)| \ll \left\{  \begin{array}{ll} q^{-(d-2)/2} \quad&\mbox{for even} ~~d\ge 4, \\
q^{-(d-1)/2} \quad&\mbox{for odd}~~ d\ge 3.\end{array}\right.$$

In this paper, we solve completely Conjecture \ref{conjzero} by showing that the first statement of the conjecture is also true. In conclusion, all sharp $L^2\to L^r$ estimates for the zero radius sphere are attained. Our main result is as follows. 
\bigskip

\begin{theorem}\label{Main2} Let $S_0$ be the sphere in $\mathbb F_{q*}^d$ with zero radius.
If  $q\equiv 3 \mod{4}$ and $d=4k+2$ for $k\in \mathbb N$, then we have
\begin{equation*} R_{S_0}^*(2 \to r) \ll 1\quad \mbox{for}~~ \frac{2d+4}{d} \le r\le \infty. \end{equation*}
\end{theorem}
A new idea is to relate our problem to the explicit Gauss sum estimate, which enables us to deduce much stronger result under the assumptions  of Theorem \ref{Main2}. We strongly expect that the $L^2 \to L^r$ extension estimates for $S_0$ will lead to further surprising applications. 


Compared to the case of the paraboloid,  we face with many difficulties in studying the restriction problem for spheres of non-zero radii. As we mentioned earlier, for the paraboloid, there is an approach to reduce the $L^2\to L^r$ extension estimates to the additive energy estimate of sets in the paraboloid. Nevertheless, such a connection has not been discovered in the case of spheres yet. Moreover, giving good bounds on the additive energy of sets on spheres is harder than the paraboloid case. The main reason is the defined equation of the paraboloid forms a graph, thus it is easy to reduce the problem to one-lower dimensional space. With these difficulties, the known results for non-zero radii spheres are much weaker than those for the paraboloid. More precisely, the following results are known for spheres with non-zero radii.
\bigskip

\begin{proposition}\label{Pro} Let $S_j$ be the sphere in $\mathbb F_{q*}^d$ with non-zero radius $j.$  Then the following results hold:
\begin{enumerate}
\item If $d\ge 3,$  then 
$ R_{S_j}^*(2\to r) \ll 1$ for $\frac{2d+2}{d-1}\le r \le \infty.$
\item If $d\ge 4$ is even, then 
$ R_{S_j}^*(p\to 4) \ll 1$ for  $\frac{12d-8}{9d-12}< r \le \infty.$
\end{enumerate}

\end{proposition}

We note that a smaller exponent gives a better extension result. The first result in Proposition \ref{Pro} was given in \cite{IK08} and it gives the Stein-Tomas result, the $L^2\to L^{(2d+2)/(d-1)}$ bound. It is well known in \cite{IK10} that the Stein-Tomas result gives the optimal $L^2\to L^r$ estimate for spheres in general odd dimensions. In even dimensions $d$, it is conjectured that the ``$r$" index of the Stein-Tomas result can be improved to $(2d+4)/d.$ 

 The second result in Proposition \ref{Pro} was proved in \cite{IK10} and it gives much better $L^p\to L^4$ extension result than the optimal result in odd dimensions.
In our next result, we significantly improve the second result of Proposition \ref{Pro}. 
\bigskip

 \begin{theorem}\label{Main1}
 Let $S_j$ be the sphere in $\mathbb F_{q*}^d$, defined as in (\ref{DefS}). If $j\ne 0$ and $d\ge 4$ is even, then we have
 \begin{equation*} R_{S_j}^*(p\to 4) \ll 1 \quad \mbox{for} ~~ \frac{4d}{3d-2}\le p \le \infty.\end{equation*}
Furthermore, the above result is sharp.
 \end{theorem}
 In next sections, we will focus on the proofs of Theorems \ref{Main2} and \ref{Main1}, and the sharpness as well. 
 \section{Sharp $L^2\to L^r$ estimates for the zero radius sphere (Theorem \ref{Main2})}
 We first prove the following weak-type restriction estimate for the sphere of zero radius, which makes a key role in proving Theorem \ref{Main2}.
\subsection{Weak-Type restriction estimates for the zero radius sphere}\label{sec3}
\begin{lemma}\label{WL2} Let $d\sigma$ denote the normalized surface measure on the sphere $S_0$ with zero radius in $\mathbb F_{q*}^d.$
In addition, assume that $d=4k+2$ for $k\in  \mathbb N$, and $q\equiv 3 \mod{4}.$ Then, for  $G\subset \mathbb F_q^d,$  we have
$$\|\widehat{1_G}\|_{L^2(S_0, d\sigma)} 
\ll \left\{  \begin{array}{ll}q^{\frac{1}{2}} |G|^{\frac{1}{2}} \quad&\mbox{for}~~q^{\frac{d+2}{2}} \le |G|\le q^{d}\\
q^{-\frac{d}{4}} |G| \quad&\mbox{for}~~ q^{\frac{d}{2}} \le |G|\le q^{\frac{d+2}{2}}\\
|G|^{\frac{1}{2}} \quad&\mbox{for}~~ 1 \le |G|\le q^{\frac{d}{2}}.\end{array}\right.
$$
\end{lemma} 

\begin{proof} By a direct comparison,  it suffices to prove the following two inequalities: for $G\subset \mathbb F_q^d,$
\begin{equation}\label{EG1}
\|\widehat{1_G}\|_{L^2(S_0, d\sigma)} \ll q^{\frac{1}{2}} |G|^{\frac{1}{2}}
\end{equation}
and
\begin{equation}\label{EG2}
\|\widehat{1_G}\|_{L^2(S_0, d\sigma)} \ll |G|^{\frac{1}{2}} + q^{\frac{-d}{4}} |G|.
\end{equation}

To prove \eqref{EG1},  observe from the Plancherel theorem that the following extension estimate holds:
$$ \|(fd\sigma)^\vee\|_{L^2(\mathbb F_q^d, dm)} = \left(\frac{|\mathbb F_q^d|}{|S_0|}\right)^{\frac{1}{2}} \|f\|_{L^2(S_0, d\sigma)} \ll q^{\frac{1}{2}} \|f\|_{L^2(S_0, d\sigma)}.$$
By duality, we obtain that 
$$ \|\widehat{g}\|_{L^2(S_0, d\sigma)} \ll q^{\frac{1}{2}}\|g\|_{L^2(\mathbb F_q^d, dm)}.$$
Now taking $g$ as the indicate function $1_G$ on $G\subset \mathbb F_q^d,$  we obtain the inequality \eqref{EG1}. 
It remains to prove the inequality \eqref{EG2}. Expanding the left-hand side of \eqref{EG2} and using the fact that $|S_0|\sim q^{d-1}$,  we see that the inequality \eqref{EG2} is equivalent to 
\begin{equation}\label{Fzero} \sum_{x\in S_0} |\widehat{1_G}(x)|^2 \ll q^{d-1}|G| + q^{\frac{d-2}{2}} |G|^2.\end{equation}

Hence, our task is to prove  this inequality. To this end, we define 
$$\nu(0):=\sum_{m, m'\in G: m-m'\in S_0}1 =\sum_{m, m'\in \mathbb F_q^d} 1_G(m) 1_G(m') 1_{S_0}(m-m').$$
Then it is clearly true that $\nu(0)\ge 0$ (in fact, we have $\nu(0)\ge |G|).$
Applying the Fourier inversion formula to the indicate function $1_{S_0}(m-m')$,  it follows that
\begin{align}\label{GS} 0\le \nu(0) &= \sum_{m, m'\in \mathbb F_q^d} 1_G(m) 1_G(m') \sum_{x\in \mathbb F_q^d}  (1_{S_0})^{\vee}(x) ~\chi(x\cdot (m-m'))\nonumber\\
&=\sum_{x\in \mathbb F_q^d} |\widehat{1_G}(x)|^2 (1_{S_0})^{\vee}(x).
\end{align}

We now apply the following consequence whose proof will be given in the next subsection.
\begin{lemma} \label{ExplicitS0}Let $S_0$ be the sphere with zero radius in $\mathbb F_q^d.$ Assume that $d=4k+2$ for $k\in \mathbb N$ and $q\equiv 3 \mod{4}.$
Then we have
$$ (1_{S_0})^{\vee}(\alpha) :=q^{-d} \sum_{y\in S_0} \chi(\alpha\cdot y) = q^{-1} \delta_0(\alpha) -q^{\frac{-(d+2)}{2}} \sum_{r\ne 0} \chi(r\|\alpha\|),$$
where $\delta_0(\alpha)=1 $ for $\alpha=(0,\ldots,0)$, and $0$ otherwise.
\end{lemma}
Inserting the formula for $(1_{S_0})^{\vee}$  into \eqref{GS}, we get 
$$0\le \sum_{x\in \mathbb F_q^d} |\widehat{1_G}(x)|^2 q^{-1} \delta_0(x) -q^{\frac{-(d+2)}{2}} \sum_{x\in \mathbb F_q^d} |\widehat{1_G}(x)|^2 \sum_{r\ne 0} \chi(r\|x\|).$$
Applying the orthogonality relation of $\chi$ to the sum over $r\ne 0$, 
$$0\le  |\widehat{1_G}(0,\ldots,0)|^2 q^{-1}  -q^{\frac{-(d+2)}{2}} (q-1) \sum_{\|x\|=0} |\widehat{1_G}(x)|^2  + q^{\frac{-(d+2)}{2}} \sum_{\|x\|\ne 0} |\widehat{1_G}(x)|^2$$
$$= q^{-1}|G|^2 -q^{\frac{-(d+2)}{2}} q\sum_{\|x\|=0} |\widehat{1_G}(x)|^2 + q^{\frac{-(d+2)}{2}} \sum_{x\in \mathbb F_q^d} |\widehat{1_G}(x)|^2$$
Since $\sum_{x\in \mathbb F_q^d} |\widehat{1_G}(x)|^2 =q^d |G|$,  solving for $\sum_{\|x\|=0} |\widehat{1_G}(x)|^2$ yields
the inequality \eqref{Fzero},
which completes the proof of Lemma \ref{WL2}.
\end{proof}

\subsection{Proof of Lemma \ref{ExplicitS0}}

We shall make use of the following fact.
\begin{lemma}[\cite{IK10}, Lemma 4] \label{FSL} Let $S_j$ be the sphere in ${\mathbb F}_q^d, d\ge 2.$
Then for any $\alpha\in {\mathbb F}_q^d,$ we have
$$ (1_{S_j})^{\vee}(\alpha)  = q^{-1} \delta_0(\alpha) + q^{-d-1}\eta^d(-1) G^d(\eta, \chi) \sum_{r \in {\mathbb F}_q^*} 
\eta^d(r)\chi\Big(jr+ \frac{\|\alpha\|}{4r}\Big),$$
where $\eta$ denotes the quadratic character of $\mathbb F_q^*$ and $G(\eta, \chi):=\sum_{s\in \mathbb F_q^*} \eta(s)\chi(s)$ which is the standard Gauss sum.
\end{lemma}

We also invoke the explicit value of the Gauss sum $G(\eta, \chi).$ 
\begin{theorem}[\cite{LN97}, Theorem 5.15]\label{ExplicitGauss}
Let $\mathbb F_q$ be a finite field with $ q= p^{\ell},$ where $p$ is an odd prime and $\ell \in {\mathbb N}.$
Then we have
$$G(\eta, \chi)= \left\{\begin{array}{ll}  {(-1)}^{\ell-1} q^{\frac{1}{2}} \quad &\mbox{if} \quad p \equiv 1 \mod{4} \\
{(-1)}^{\ell-1} i^\ell q^{\frac{1}{2}} \quad &\mbox{if} \quad p\equiv 3 \mod{4}.\end{array}\right. $$
\end{theorem}
\paragraph{Proof of Lemma \ref{ExplicitS0}.}
Combining the assumption that $d=4k+2$ for $k\in \mathbb N$ and Lemma \ref{FSL}, we have  
  \begin{equation*}\label{oneG}
(1_{S_0})^{\vee}(\alpha)  = q^{-1} \delta_0(\alpha) + q^{-d-1} G^d(\eta, \chi) \sum_{r \in {\mathbb F}_q^*} 
\chi\Big(\frac{\|\alpha\|}{4r}\Big).
\end{equation*}
 
It is clear that $q\equiv 3\mod{4}$ if and only if $q=p^\ell$ for some prime $p\equiv 3 \mod{4}$ and odd integer $\ell.$ Hence it follows from Theorem \ref{ExplicitGauss} that  $G^d(\eta, \chi)=-q^{\frac{d}{2}}$ provided that  $q\equiv 3\mod{4}$ and $d=4k+2.$  
Now  the statement of Lemma \ref{ExplicitS0} follows by a simple change of variables (namely, we replace $1/4r$ by  r).   $\hfill\square$   

We are now ready to give a proof of Theorem \ref{Main2}.
\subsection{Proof of Theorem \ref{Main2}}
Assume that $q\equiv 3\mod{4}$ and $d=4k+2$ for $k\in \mathbb N.$ 
By the nesting properties of $L^p$-norms,  we only need to prove the following extension estimate for $S_0:$
\begin{equation*}\label{Red2k} \|(fd\sigma)^\vee\|_{L^{\frac{2d+4}{d}}(\mathbb F_q^d, dm)} \ll \|f\|_{L^{2}(S_0, d\sigma)}.\end{equation*}
By duality, it suffices to establish the following restriction estimate for $S_0$:
\begin{equation*}
\|\widehat{g}\|_{L^2(S_0, d\sigma)} \ll \|g\|_{L^{\frac{2d+4}{d+4}}(\mathbb F_q^d, dm)} := \left( \sum_{m\in \mathbb F_q^d} |g(m)|^{\frac{2d+4}{d+4}} \right)^{\frac{d+4}{2d+4}}.
\end{equation*}
Normalizing the functions $g$, it suffices to show that 
$\|\widehat{g}\|_{L^2(S_0, d\sigma)} \ll 1$
for all functions $g$ on $\mathbb F_q^d$ such that 
\begin{equation}\label{kohgo}\sum\limits_{m\in \mathbb F_q^d} |g(m)|^{\frac{2d+4}{d+4}}=1.\end{equation}
Without loss of generality, we may assume that the function $g$ takes the following form: for $i=0,1, \cdots, L\ll \log{q},$
\begin{equation}\label{Ass2kN} g=\sum_{i=0}^L 2^{-i} 1_{G_i},\end{equation}
where $\{G_i\}$ is a collection of pairwise disjoint subsets of $\mathbb F_q^d.$

From \eqref{kohgo} and \eqref{Ass2kN}, we can also assume that 
$ \sum_{i=0}^L 2^{-\frac{2d+4}{d+4}i} |G_i| =1.$ 
This clearly implies that
\begin{equation}\label{GG}  |G_i|\le 2^{\frac{2d+4}{d+4} i} \quad \mbox{for all}~~i=0,1,2, \ldots, L. \end{equation}

In summary, our task is to prove that
$
\|\widehat{g}\|_{L^2(S_0, d\sigma)} \ll 1
$ for any function $g$ such that for $L \sim \log{q}$ and  a collection $\{G_i\}$ of pairwise disjoint subsets of $\mathbb F_q^d,$ one can write 
$ g=\sum_{i=0}^L 2^{-i} 1_{G_i}$ with \eqref{GG}.

We have
\begin{align*} &\|\widehat{g}\|_{L^2(S_0, d\sigma)} \le \sum_{i=0}^L 2^{-i} \|\widehat{1_{G_i}}\|_{L^2(S_0, d\sigma)}\\
 \ll&  \sum_{\substack{0\le i\le L\\ 1\le 2^{\frac{2d+4}{d+4} i}\le q^{\frac{d}{2}}}} 2^{-i} \|\widehat{1_{G_i}}\|_{L^2(S_0, d\sigma)} + 
\sum_{\substack{0\le i\le L\\ q^{\frac{d}{2}}\le 2^{\frac{2d+4}{d+4} i}\le q^{\frac{d+2}{2}}}} 2^{-i} \|\widehat{1_{G_i}}\|_{L^2(S_0, d\sigma)} + 
\sum_{\substack{0\le i\le L\\ q^{\frac{d+2}{2}}\le 2^{\frac{2d+4}{d+4} i}\le q^{d}}} 2^{-i} \|\widehat{1_{G_i}}\|_{L^2(S_0, d\sigma)} \\
=& \mbox{I} + \mbox{II} +\mbox{III}. \end{align*}

 Utilizing Lemma \ref{WL2} with \eqref{GG}, we get
 $$ \mbox{I} \ll \sum_{\substack{0\le i\le L\\ 1\le 2^{\frac{2d+4}{d+4} i}\le q^{\frac{d}{2}}}} 2^{-i} |G_i|^{\frac{1}{2}} 
 \ll \sum_{\substack{0\le i\le L\\ 1\le 2^{\frac{2d+4}{d+4} i}\le q^{\frac{d}{2}}}} 2^{-i}2^{\frac{d+2}{d+4} i}  \ll 1,$$
 
 $$ \mbox{II} \ll \sum_{\substack{0\le i\le L\\q^{\frac{d}{2}}\le 2^{\frac{2d+4}{d+4} i}\le q^{\frac{d+2}{2}}}} 2^{-i} q^{-\frac{d}{4}}|G_i|
 \ll q^{-\frac{d}{4}}\sum_{\substack{0\le i\le L\\ q^{\frac{d}{2}}  \le 2^{\frac{2d+4}{d+4} i}\le q^{\frac{d+2}{2}}}} 2^{-i}2^{\frac{2d+4}{d+4} i}  \ll q^{-\frac{d}{4}} q^{\frac{d}{4}} = 1$$
 and 
 
 $$\mbox{III} \ll  \sum_{\substack{0\le i\le L\\ q^{\frac{d+2}{2}}\le 2^{\frac{2d+4}{d+4} i}\le q^{d}}} 2^{-i} q^{\frac{1}{2}} |G_i|^{\frac{1}{2}} 
 \ll  q^{\frac{1}{2}} \sum_{\substack{0\le i\le L\\q^{\frac{d+2}{2}}\le 2^{\frac{2d+4}{d+4} i}\le q^{d}}} 2^{-i}2^{\frac{d+2}{d+4} i}  \ll q^{\frac{1}{2}}q^{-\frac{1}{2}}=1.$$

Thus, we complete the proof. $\hfill\square$

\section{Sharp $L^p\to L^4$ extension estimate for spheres of non-zero radii (Theorem \ref{Main1})}\label{sec4}
One of key steps in the proof of Theorem \ref{Main1} is the following lemma on an additive energy estimate of a set on spheres. 
\bigskip

 \begin{lemma}\label{thm2}
Let $S_j$ be the sphere  with non-zero radius $j\ne 0$ in $\mathbb{F}_q^d$ with $d\ge 4$ even. For $A\subset S_j$, let $E(A)$ be the number of tuples $(a, b, c, d)\in A^4$ such that 
$a+b=c+d.$
Then we have 
\[E(A)\ll \frac{|A|^3}{q}+q^{\frac{d-2}{2}}|A|^2.\]
\end{lemma}
We will see later in subsection \ref{chale} that spheres $S_j$ actually contain an affine subspace $H$ of dimension $\frac{d-2}{2}$. Therefore, by taking $A=H$, we have $E(A)=|A|^3=q^{\frac{d-2}{2}}|A|^2$. In other words, Lemma \ref{thm2} is optimal up to a constant factor. 

Note that over prime fields, Rudnev \cite{Ru18} derived the following estimate for sets on non-zero radius spheres in only three and four dimensions. 
$$ E(A) \ll \frac{|A|^3}{q}+|A|^{\frac{5}{2}} + |A|k_0^2 + |A|^2 k_0,$$
where $k_0$ denotes the maximal number of points of $A$ lying on an isotropic line.

The main challenge in the proof of Lemma \ref{thm2} is to give an optimal upper bound on the number of pairs $(x, y)\in A\times A$ such that $||x-y||=0$. Since $A$ is a set on $S_j$, the equation $||x-y||=0$ is equivalent with $x\cdot y=j$. When $d$ is small, the problem can be handled by combinatorial arguments. However, when $d$ is large, the problem becomes difficult. 

It is not hard to construct a set $A\subset S_j$ with many pairs of zero distance. Indeed, suppose $d=4k+2$, and the subspace $\mathbb{F}_{q}^{d-2}\times \{0\}\times \{0\}$ contains $\frac{d-2}{2}$ vectors that are mutually orthogonal (see \cite{hart} for a detailed proof). Let $v_1, \ldots, v_{\frac{d-2}{2}}$ be such vectors. If we change the last coordinate of each $v_i$ from $0$ to $j$ for all $1\le i\le (d-2)/2$, then we have $||v_i||=j$ and $||v_i-v_{i'}||=0$ for all $i\ne i'$. Let $A$ be the subspace spanned by these changed $v_i$, $1\le i\le \frac{d-2}{2}$, then we have $|A|=q^{\frac{d-2}{2}}$ and the number of pairs $(x, y)\in A\times A$ such that $||x-y||=0$ is $|A|^2$. 

Thanks to the following result due to the first and second listed authors in \cite{IK10}, we have an optimal bound on the number of pairs $(x, y)\in A\times A$ such that $||x-y||=0$. 
\bigskip

\begin{lemma}[\cite{IK10}, Theorem 5] \label{LemIK}
Let $S_j$ be the sphere in $\mathbb F_q^d.$
If $d\ge 4$ is even and $A$ is any subset of $S_j$ with $j\ne 0$, then we have
$$ \sum_{x, y\in A: x\cdot y=j} 1 \ll \frac{|A|^2}{q} + q^{\frac{d-2}{2}} |A|.$$
\end{lemma}

The following lemmas are also important in the proof of Lemma \ref{thm2}. 

\begin{lemma}\label{lm55x0}
Let $\cal{P}$ be a paraboloid in $\mathbb{F}_q^{d+1}$. Suppose that $\cal{A}$ and $\cal{B}$ are subsets in $P$ with the property that there are no two points $x=(x', ||x'||)$ and $y=(y', ||y'||)$ in $\cal{A}$ such that $x'=\lambda y'$ with $\lambda \ne 0, 1$. Then the number of pairs $(\alpha, \beta)\in \cal{A}\times \cal{B}$ with $\alpha+\beta\in \cal{P}$ is bounded by $\sim$
\[\frac{|\cal{A}||\cal{B}|}{q}+q^{\frac{d-1}{2}}|{\cal A}|^{1/2}
|{\cal B}|^{1/2}.\]
\end{lemma}
\begin{proof}
For $\alpha\in \cal{A}$ and $\beta\in \cal{B}$, we write $\alpha=({\alpha}', ||{\alpha}'||)$ and $\beta=({\beta}', ||{\beta}'||)$.

 It is clear that if $\alpha+\beta\in \cal{P}$, then we have 
\[||{\alpha}'+{\beta}'||=||{\alpha}'||+||{\beta}'||.\]
This implies that ${\alpha}'\cdot {\beta}'=0$. 

Note that ${\alpha}', {\beta}'\in \mathbb{F}_q^d$.  Let ${\cal A}'$ and ${\cal B}'$ be the projections of $\cal{A}$ and $\cal{B}$ on $\mathbb{F}_q^d$, respectively. Let $N$ be the number of pairs $({\alpha}', {\beta}')\in {\cal A}'\times {\cal B}'$ such that ${\alpha}'\cdot {\beta}'=0$. 

For each ${\alpha}'\in {\cal A}'$, we assign it with a hyperplane in $\mathbb{F}_q^d$ defined by the equation ${\alpha}'\cdot x'=0$. Let $H$ be the set of all corresponding hyperplanes in $\mathbb F_q^d$. It follows from the assumption of the lemma that all hyperplanes in $H$ are distinct. So we have $|H|=|{\cal A}'|$.

It is clear that $N$ is equal to the number of incidences between hyperplanes in $H$ and points in ${\cal B}'$, denoted by $I({\cal B}', H)$. One can apply a point-plane incidence theorem due to Vinh in \cite{vinhajm} to achieve the following
\[I({\cal B}', H)\le \frac{|{\cal B}'||H|}{q}+q^{\frac{d-1}{2}}\sqrt{|{\cal B}'||H|}=\frac{|{\cal A}'||{\cal B}'|}{q}+q^{\frac{d-1}{2}}\sqrt{|{\cal A}'||{\cal B}'|}.\]
Using the fact that $|{\cal A}'|=|\cal{A}|$ and $|{\cal B}'|=|{\cal B}|$, the lemma follows. 
\end{proof}
\bigskip

\begin{lemma}\label{lm55x1}
Suppose that $S_j$ is the sphere centered at the origin of radius $j\ne 0$. Then there are no three distinct points $x, y, z\in S_j$ such that 
\[x-z=\lambda (y-z)\]
with $\lambda \ne 0, 1$ and $||y-z||\ne 0$.
\end{lemma}
\begin{proof}
Indeed, if there exist such points, this means that $x, y, z$ lie on both a line and $S_j$, and the norm of its direction vector is non-zero. This is impossible, because the line intersects $S_j$ in at most $2$ points. 
\end{proof}
\bigskip

\paragraph{Proof of Lemma \ref{thm2}.} Since $A\subset S_j,$ it is clear that the number of tuples $(x, y, z, t)\in A^4$ with $x+y=z+t$ is at most  the number of triples $(x, y, z)\in A^3$ such that $x+y-z\in S_j$. In other words, if $A\subset S_j,$ then
\begin{equation}\label{easyE} E(A)\le  \sum_{x\in A} \sum_{y\in A} \sum_{z\in A}  1_{S_j}(x+y-z).\end{equation}
It is observed in \cite{Ru18} that if $x,y,z\in A\subset S_j$ with $x+y-z\in S_j,$ then  $(x-z)\cdot (y-z)=0.$ Indeed, suppose that $x,y,z\in A\subset S_j$ and $x+y-z\in S_j.$ Then  it satisfies that
\[||x+y-z||=j.\]
Namely we have 
$$ \|x\|+ \|y\|+ \|z\| +2x\cdot y-2y\cdot z-2x\cdot z=j,$$
which is equivalent to
\[||x||+||y||+||z||+2y\cdot(x-z)-2x\cdot z=j.\]
Since $\|x\|=\|y\|=\|z\| =j,$ we obtain 
$$ y\cdot(x-z)-x\cdot z=-j,$$
which can be rewritten by 
$$ y\cdot(x-z)-x\cdot z=-z\cdot z,$$ because $z\cdot z=j$ for $z\in S_j.$
Namely, it satisfies that
\[(x-z)\cdot (y-z)=0.\]

From \eqref{easyE} and this fact, in order to complete the proof of  Lemma \ref{thm2}, it will be enough to show that
$$ \sum_{x,y,z\in A: (x-z)\cdot (y-z)=0} 1 \ll \frac{|A|^3}{q}+q^{\frac{d-2}{2}}|A|^2.
$$ 
We now consider two cases:

{\bf Case $1$:} We count the number of triples $(x, y, z)\in A^3$ with the desired property such that  $||x-z||=0$ or $||y-z||=0$. 

Since $\|a-b\|=2j-2a\cdot b$ for $a,b\in A\subset  S_j,$ it follows from Lemma \ref{LemIK} that the number of pairs $(a, b)\in A^2$ with $||a-b||=0$ is bounded by $\sim$ $(|A|^2/q+q^{(d-2)/2}|A|)$. 
Thus, the number of such triples is bounded by the quantity similar to
\[\frac{|A|^3}{q}+q^{\frac{d-2}{2}}|A|^2.\]

{\bf Case $2$:} For each fixed $z\in A$, we shall count the number of triples $(x, y, z)\in A^3$ with the desired property that $||x-z||\ne 0$ and $||y-z||\ne 0$. 

For a fixed $z\in A,$ define \[X:=\{(x-z, ||x-z||)\colon x\in A, ||x-z||\ne 0\},\] 
and \[Y:=\{(y-z, ||y-z||)\colon y\in A, ||y-z||\ne 0\}\]
as subsets in a paraboloid in $\mathbb{F}_q^{d+1}$.

We observe that if $(x-z)\cdot (y-z)=0$, then we have 
\[(x-z, ||x-z||)+(y-z, ||y-z||)\in {\cal P},\]
where ${\cal P}$ denotes the paraboloid in $\mathbb F_q^{d+1}.$
Thus, we may reduce the problem to counting the number of pairs $(\alpha, \beta)\in X\times Y$ such that $\alpha+\beta \in {\cal P}$. We denote this number by $T(X, Y)$.
We also see that the condition $(x-z)\cdot (y-z)=0$ implies $(x-z)\cdot (\lambda y-\lambda z)=0$ for any $\lambda \ne 0, 1$. 
It follows from Lemma \ref{lm55x1} that there are no two pairs $(x,z), (y, z)\in A^2$ such that $x-z=\lambda (y-z)$ with $\lambda \ne 0,1$. Thus, for each fixed $z\in A,$ we can define a new set 
\[Y':=\{(\lambda (y-z), \lambda^2||y-z||)\colon y\in A, ||y-z||\ne 0, \lambda \ne 0, 1 \}.\]

It is clear that $|Y'|=(q-2)|Y|$ and $(q-2)T(X, Y)= T(X, Y')$. Note that the set $Y'$ satisfies the condition of Lemma \ref{lm55x0}. Therefore,  we can apply Lemma \ref{lm55x0} to obtain that 
\[T(X, Y')\ll \frac{(q-2)|A|^2}{q}+q^{\frac{d-1}{2}}|A|^{1/2}|Y'|^{1/2},\]
where we have used the fact that $|X|=|Y|\le |A|.$
Hence, $T(X, Y)\ll \frac{|A|^2}{q}+q^{\frac{d-2}{2}}|A|.$
Summing over all $z\in A$, we obtain the desired estimate. Thus, the proof of Lemma \ref{thm2} is complete.$\hfill\square$

As a consequence of Lemma \ref{thm2}, we have the  following weak-type $L^4$ extension estimate, which will be used in the proof of Theorem \ref{Main1}.
\begin{lemma}\label{WeakL4k}
Let $d\sigma$ denote the normalized surface measure on the sphere $S_j \subset \mathbb F_q^d$ with non-zero radius $j\ne 0.$ If $d\ge 4$ is even and $A\subset S_j,$ then we have
$$ \|(1_Ad\sigma)^\vee\|_{L^{4}(\mathbb F_q^d, dm)} 
\ll \left\{  \begin{array}{ll} q^{\frac{-3d+3}{4}} |A|^{\frac{3}{4}} \quad&\mbox{for}~~q^{\frac{d}{2}} \le |A|\ll q^{d-1}\\
q^{\frac{-5d+6}{8}} |A|^{\frac{1}{2}} \quad&\mbox{for}~~ q^{\frac{d-2}{2}} \le |A|\le q^{\frac{d}{2}}\\
q^{\frac{-3d+4}{4}}|A|^{\frac{3}{4}} \quad&\mbox{for}~~ 1 \le |A|\le q^{\frac{d-2}{2}},\end{array}\right.
$$
\end{lemma}
\begin{proof} By expanding $\|(1_Ad\sigma)^\vee\|_{L^{4}(\mathbb F_q^d, dm)}$ and applying the orthogonality property of $\chi$, we see 
$$ \|(1_Ad\sigma)^\vee\|_{L^{4}(\mathbb F_q^d, dm)} =\frac{q^{\frac{d}{4}}}{|S_j|} E(A)^{1/4} \sim q^{\frac{-3d+4}{4}}E(A)^{1/4},$$ 
where $E(A)=\sum\limits_{a,b,c,d\in A: a+b=c+d} 1.$ When $1\le |A|\le q^{\frac{d-2}{2}}$, we use the trivial estimate that $E(A)\le |A|^3.$ On the other hand, when $ q^{\frac{d-2}{2}} \le |A|\ll q^{d-1},$ applying Lemma \ref{thm2} yields the statement of Lemma \ref{WeakL4k} by a direct computation.
\end{proof}    
We are now ready to give a proof of Theorem \ref{Main1}.
\paragraph{Proof of Theorem \ref{Main1}.}
By the nesting properties of $L^p$-norms, it suffices to establish the following  extension estimate:
\begin{equation}\label{Red1} \|(fd\sigma)^\vee\|_{L^{4}(\mathbb F_q^d, dm)} \ll \|f\|_{L^{4d/(3d-2)}(S_j, d\sigma)}\sim \left(q^{-d+1} \sum_{x\in S_j} |f(x)|^{\frac{4d}{3d-2}}\right)^{\frac{3d-2}{4d}}.\end{equation}

To prove \eqref{Red1}, by normalizing functions $f$, it will be enough to prove 
$$\Omega:=q^{\frac{3d^2-5d+2}{4d}} \|(fd\sigma)^\vee\|_{L^{4}(\mathbb F_q^d, dm)} \ll 1,$$
assuming that
\begin{equation}\label{Ass1k} \sum_{x\in S_j} |f(x)|^{\frac{4d}{3d-2}}=1.\end{equation}
Without loss of generality, we may assume that the function $f$ can be decomposed as follows: for $k=0,1, \cdots, L\ll \log{q},$
\begin{equation}\label{Ass2k} f=\sum_{k=0}^L 2^{-k} 1_{A_k},\end{equation}
where $\{A_k\}$ is a collection of pairwise disjoint subsets of $S_j.$
From \eqref{Ass1k} and \eqref{Ass2k}, we have  
$$ \sum_{k=0}^L 2^{-\frac{4d}{3d-2}k} |A_k| =1,$$ 
This implies that  
\begin{equation}\label{Ass3k} |A_k|\le  2^{\frac{4d}{3d-2}k}\quad \mbox{for all}~~ k=0,1,2,\ldots.\end{equation}
It follows that
$$\Omega=q^{\frac{3d^2-5d+2}{4d}} \|(fd\sigma)^\vee\|_{L^{4}(\mathbb F_q^d, dm)} \le q^{\frac{3d^2-5d+2}{4d}} \sum_{k=0}^L 2^{-k} \|(1_{A_k}d\sigma)^\vee\|_{L^{4}(\mathbb F_q^d, dm)}$$
$$\ll  q^{\frac{3d^2-5d+2}{4d}}\sum_{\substack{0 \leq k \leq L\\ 1\le 2^{\frac{4d}{3d-2}k} \leq  q^{\frac{d-2}{2}}}} 2^{-k} \|(1_{A_k}d\sigma)^\vee\|_{L^{4}(\mathbb F_q^d, dm)}$$
 $$+  q^{\frac{3d^2-5d+2}{4d}}\sum_{\substack{0 \leq k \leq L\\ 
 q^{\frac{d-2}{2}} \le 2^{\frac{4d}{3d-2}k} \leq  q^{\frac{d}{2}}}} 2^{-k} \|(1_{A_k}d\sigma)^\vee\|_{L^{4}(\mathbb F_q^d, dm)}$$
 $$+q^{\frac{3d^2-5d+2}{4d}}\sum_{\substack{0 \leq k \leq L\\ 
 q^{\frac{d}{2}}\le 2^{\frac{4d}{3d-2}k} \ll  q^{d-1}}} 2^{-k} \|(1_{A_k}d\sigma)^\vee\|_{L^{4}(\mathbb F_q^d, dm)}$$
 $$=: \Omega_1 + \Omega_2 + \Omega_3.
 $$ 

Applying Lemma \ref{WeakL4k} with \eqref{Ass3k}, we have
$$ \Omega_1 \ll q^{\frac{-d+2}{4d}} \sum_{\substack{0 \leq k \leq L\\ 1\le 2^{\frac{4d}{3d-2}k} \leq  q^{\frac{d-2}{2}}}} 2^{-k} |A_k|^{\frac{3}{4}} 
\ll q^{\frac{-d+2}{4d}} \sum_{\substack{0 \leq k \leq L\\ 1\le 2^{\frac{4d}{3d-2}k} \leq  q^{\frac{d-2}{2}}}} 2^{\frac{2}{3d-2}k} \ll q^{\frac{-d+2}{4d}} q^{\frac{d-2}{4d}}=1,$$

\begin{align*} \Omega_2 &\ll q^{\frac{d^2-4d+4}{8d}}  \sum_{\substack{0 \leq k \leq L\\ q^{\frac{d-2}{2}} \le 2^{\frac{4d}{3d-2}k} \leq  q^{\frac{d}{2}}}} 2^{-k} |A_k|^{\frac{1}{2}} \\
&\ll q^{\frac{d^2-4d+4}{8d}}  \sum_{\substack{0 \leq k \leq L\\ q^{\frac{d-2}{2}} \le 2^{\frac{4d}{3d-2}k} \leq  q^{\frac{d}{2}}}} 2^{\frac{-d+2}{3d-2}k} \ll q^{\frac{d^2-4d+4}{8d}} q^{\frac{-d^2+4d-4}{8d}} =1,\end{align*}
and 
$$\Omega_3 \ll q^{\frac{-d+1}{2d}}\sum_{\substack{0 \leq k \leq L\\ q^{\frac{d}{2}}\le 2^{\frac{4d}{3d-2}k} \ll  q^{d-1}}} 2^{-k} |A_k|^{\frac{3}{4}} 
\ll  q^{\frac{-d+1}{2d}}\sum_{\substack{0 \leq k \leq L\\ q^{\frac{d}{2}}\le 2^{\frac{4d}{3d-2}k} \ll  q^{d-1}}} 2^{\frac{2}{3d-2}k} \ll  q^{\frac{-d+1}{2d}}  q^{\frac{d-1}{2d}} =1. $$
Thus, the proof is complete. $\hfill\square$

\subsection{Sharpness of Theorem \ref{Main1}}\label{chale}
To prove the sharpness, it will be enough from the second necessary condition in \eqref{Necessary2} to show that the sphere $S_j$ with $j\ne 0$ contains an affine subspace $H$ with $|H|=q^{(d-2)/2}$ for even $d\ge 4.$
It is well known from \cite[P.79]{Gr02} or \cite[Theorem 1]{AA16} that  if $d$ is even then  the sphere $S_j$ can be equivalently classified as the following variety:
\begin{equation}\label{TV}\widetilde{S_j}:=\{x\in \mathbb F_q^d: x_1^2-x_2^2+ \cdots+ x_{d-3}^2-x_{d-2}^2 + x_{d-1}^2- \alpha x_d^2=j\},\end{equation}
where $\alpha$ takes 1 or a fixed non-square number of $\mathbb F_q^*$ which is determined by $1=\eta((-1)^{d/2})~ \eta(\alpha)$ for the quadratic character $\eta$ of $\mathbb F_q^*.$ Hence  $\widetilde{S_j}$ contains a $(d-2)/2$-dimensional affine subspace
$$\left\{(t_1,t_1, \cdots, t_{(d-2)/2}, t_{(d-2)/2}, a, b)\in \mathbb F_q^d: t_j\in \mathbb F_q, j=1,2,\ldots, (d-2)/2\right\},$$
where $(a,b)\in \mathbb F_q^2$ is a solution to the equation $a^2-\alpha b^2=j\ne 0.$ Thus, $S_j$ also contains a $\frac{d-2}{2}$-dimensional affine subspace. This completes the sharpness of Theorem \ref{Main1}.

\section{Appendix} \label{Ap4}
In this appendix, we provide a proof of  the duality of extension and restriction estimates for a variety $(V, d\sigma)$ in $(\mathbb F_{q*}^d, dx).$
More precisely we prove the following theorem whose proof is well known in the Euclidean setting.
\begin{theorem} [Duality of extension and restriction estimates]\label{ThmDual}
Let $1\le p, r\le \infty$ and $C$ be a constant. Then the following two statements are equivalent:
\begin{enumerate}
\item
For all functions $f: V\to \mathbb C,$ we have 
$\| (fd\sigma)^{\vee}\|_{L^r(\mathbb F_q^d, dm)} \le C \|f\|_{L^p(V, d\sigma)}.$
\item  For all functions $g: \mathbb F_q^d \to \mathbb C,$ we have
$\| \widehat{g}\|_{L^{p'}(V, d\sigma)} \le C \|g\|_{L^{r'}(\mathbb F_q^d, dm)}.$
\end{enumerate}
\end{theorem}
 
To prove the theorem, we need some definitions and notations.
Given functions $f_1, f_2: (V, d\sigma) \to \mathbb C$ and $g_1, g_2: (\mathbb F_q^d, dm) \to \mathbb C,$ we define
$$ <f_1, f_2>_{(V, d\sigma)}:= \int_{V} f_1(x) \overline{f_2(x)} ~ d\sigma(x),$$
$$ <g_1, g_2>_{(\mathbb F_q^d, dm)} := \int_{\mathbb F_q^d} g_1(m) \overline{g_2(m)} ~dm.$$
By a direct computation, it is not hard to prove that
\begin{equation}\label{InnerP} <f, \widehat{g}>_{(V, d\sigma)} = < (fd\sigma)^{\vee}, g >_{(\mathbb F_q^d, dm)}.\end{equation}

We also need the following lemma.

\begin{lemma}\label{l3.3} Let $1\leq p \le \infty.$  Then we have
\begin{enumerate} 
\item $\displaystyle\|f\|_{L^p(V, d\sigma)}=\sup_{\|g\|_{L^{p'}(V, d\sigma)}\le 1} | < f, g>_{(V, d\sigma)}|.$
\item
$ \displaystyle\|g\|_{L^p(\mathbb F_q^d, dm)} =  \sup_{\|f\|_{L^{p'}(\mathbb F_q^d, dm)}\le 1} | < g, f>_{(\mathbb F_q^d, dm)}|                         .$
\end{enumerate}
\end{lemma}
\begin{proof}
\textbf{(Proof of the first conclusion)} By H\"{o}lder's inequality, we have
$$ \displaystyle\|f\|_{L^p(V, d\sigma)}\ge\sup_{\|g\|_{L^{p'}(V, d\sigma)}\le 1} | < f, g>_{(V, d\sigma)}|.
$$
 Now, we prove the following reverse inequality:
 $$\displaystyle\|f\|_{L^p(V, d\sigma)}\le\sup_{\|g\|_{L^{p'}(V, d\sigma)}\le 1} | < f, g>_{(V, d\sigma)}|.$$
Since there is nothing to prove in the case that $f\equiv 0$, we can assume $f(x)\neq 0$ for some $x$ in $V.$
We consider three cases. For simplicity, the notation $(V ,d\sigma)$ is omitted from now on.

\begin{itemize}
  \item [$\mathbf{Case \,1.}$] If $p=1,$ then $p'=\infty.$ In this case,  we take
$$
g(x):= 
 \left\{
 \begin{array}{cc}
 \displaystyle\frac{{f(x)}}{|f(x)|} & \mbox{if}  \quad f(x) \neq0\\
0 & \mbox{if} \quad f(x)=0.
\end{array}
\right.
$$

Then $\|g\|_{L^\infty}=1$ and $|<f, g>|=\|f\|_{L^1}.$

  \item [$\mathbf{Case \,2.}$] If $1< p < \infty$, then $1< p'=p/(p-1) <\infty.$ We choose
$$
g(x):=\frac{|f(x)|^{p-2}{f(x)}}{\|f\|_{L^p}^{p-1}}
$$
It is not hard to check that  $\|g\|_{L^{p'}}=1$ and $ | < f, g>|=\|f\|_{L^p}.$





 \item [$\mathbf{Case \,3.}$] If  $p=\infty,$ then $p'=1.$ Let $x_0$ be an element in $V$ such that 
 $$|f(x_0)|=\max\limits_{x\in V} |f(x)|.$$ 
 In other words, $|f(x_0)|=||f||_{\infty}.$
We set
$
g(x):= 
 \left\{
 \begin{array}{cc}
 \displaystyle\frac{{f(x_0)}|V|}{|f(x_0)|} & \mbox{if}  \quad x=x_0\\
0 & \mbox{if} \quad x\ne x_0.
\end{array}
\right.
$

Then we have $\|g\|_{L^1}=1$ and $| < f, g>|=|f(x_0)|=\|f\|_{L^\infty}.$
\end{itemize}
We have completed the proof of the first part of the lemma.  The second part of  Lemma \ref{l3.3} is almost identical with that of the first part except we  use  the counting measure $dm$ and $\mathbb F_q^d$ in the place of $d\sigma$ and $V.$  In fact, Lemma \ref{l3.3} holds for any measure spaces in the finite field setting.
\end{proof}

We are ready to prove Theorem \ref{ThmDual}.\\
\textbf{(Case 1)} Assume that the first part of Theorem \ref{ThmDual} holds.
Then, using the first conclusion of Lemma \ref{l3.3}, the equality \eqref{InnerP},  H\"{o}lder's inequality, and our assumption, we   obtain the statement of the second part of Theorem \ref{ThmDual} as follows:
\begin{align*} \| \widehat{g}\|_{L^{p'}(V, d\sigma)} &=\sup_{||f||_{L^p(V, d\sigma)} \le 1} |<f, \widehat{g}>_{(V, d\sigma)}|
=\sup_{||f||_{L^p(V, d\sigma)} \le 1} |<(fd\sigma)^{\vee}, g>_{(\mathbb F_q^d, dm)}|\\
&\le \sup_{||f||_{L^p(V, d\sigma)} \le 1} || (fd\sigma)^{\vee}||_{L^r(\mathbb F_q^d, dm)} ||g||_{L^{r'}((\mathbb F_q^d, dm)} \\
&\le \sup_{||f||_{L^p(V, d\sigma)} \le 1} C ||f||_{L^p(V, d\sigma)} ||g||_{L^{r'}((\mathbb F_q^d, dm)} \le C||g||_{L^{r'}((\mathbb F_q^d, dm)}. 
\end{align*}
\textbf{(Case 2)} Assume that the second part of Theorem \ref{ThmDual} is statisfied. 
In this case, we use the second conclusion of Lemma \ref{l3.3},  the equality \eqref{InnerP},  H\"{o}lder's inequality, and our assumption. Then the statement of the first part of Theorem \ref{ThmDual} is proven as follows:
\begin{align*} \| (fd\sigma)^{\vee}\|_{L^{r}(\mathbb F_q^d, dm)} &=\sup_{||g||_{L^{r'}(\mathbb F_q^d, dm)} \le 1} 
|<g, (fd\sigma)^{\vee}>_{(\mathbb F_q^d, dm)}|
=\sup_{||g||_{L^{r'}(\mathbb F_q^d, dm)} \le 1} |< \widehat{g}, f>_{(V, d\sigma)}|\\
&\sup_{||g||_{L^{r'}(\mathbb F_q^d, dm)} \le 1} || \widehat{g}||_{L^{p'}(V, d\sigma)} ||f||_{L^{p}(V, d\sigma)} \\
&\le \sup_{||g||_{L^{r'}(\mathbb F_q^d, dm)} \le 1} C ||g||_{L^{r'}(\mathbb F_q^d, dm)}||f||_{L^p(V, d\sigma)}  \le C||f||_{L^p(V, d\sigma)}, 
\end{align*}
as desired.

\section*{Acknowledgements}
Authors thank anonymous reviewers for the valuable comments which helped us improve the quality of our paper.\\

A. Iosevich was partially supported by the NSA Grant H98230-15-1-0319. D. Koh was supported by Basic Science Research Program through the National
Research Foundation of Korea(NRF) funded by the Ministry of Education, Science
and Technology(NRF-2018R1D1A1B07044469). T. Pham was supported by Swiss National Science Foundation grant P400P2-183916. Chun-Yen Shen was supported in part by MOST, through grant 108-2628-M-002-010-MY4.


 \bibliographystyle{amsplain}

\end{document}